\documentclass{article}
\usepackage{bbm}
\usepackage{cancel}

\usepackage{comment}

\usepackage{geometry} 

\usepackage{amsfonts, amsmath, amsthm, amssymb, latexsym, graphicx, mathrsfs, tikz, cancel, cleveref, authblk, soul}

\newcommand{\intd}{\, \textnormal{d}}
\newcommand{\ud}{\textnormal{d}}

\newcommand{\mbb}[1]{\mathbb{#1}}

\newcommand{\inprod}[2]{\langle{#1},{#2}\rangle}

\newcommand{\cond}[1]{\, #1 \vert \,}

\newcommand{\Pbb}{\mathbb P}
\newcommand{\R}{\mathbb R}
\newcommand{\N}{\mathbb N}

\theoremstyle{plain}
\newtheorem{theorem}{Theorem}[section]

\newtheorem{lemma}[theorem]{Lemma}

\theoremstyle{definition}

\newtheorem{remark}[theorem]{Remark}

\theoremstyle{remark}

\numberwithin{equation}{section}

\title{Collision times of multivariate Bessel processes with their Weyl chambers' boundaries and their Hausdorff dimension}
\author[1]{Nicole Hufnagel\thanks{nicole.hufnagel@hhu.de}}
\author[2]{Sergio Andraus\thanks{sergio@tsukuba-g.ac.jp}}

\affil[1]{Heinrich Heine University Düsseldorf}
\affil[2]{Tsukuba Gakuin University}

\date{\today}

\begin{document}
	\maketitle
	\allowdisplaybreaks
	\abstract{
		Multivariate Bessel processes, otherwise known as radial Dunkl processes, are stochastic processes defined in a Weyl chamber that are repelled from the latter's boundary by a singular drift with a strength given by the multiplicity function $k$. It is a well-known fact that when $k$ is sufficiently small, these processes hit the Weyl chamber's boundary almost surely, and it was recently shown by the authors that the collision times for the process of type $A$, also known as the Dyson model (a one-dimensional multiple-particle stochastic system), have a Hausdorff dimension that depends on $k$. In this paper, we use the square of the alternating polynomial, which corresponds to the reflection group of the process, to extend this result to all multivariate Bessel processes of rational type, and we show that the Hausdorff dimension of collision times is a piecewise-linear function of the minimum of $k$, but is independent of the dimension of the space where the process lives. This implies that the Hausdorff dimension is independent of the particle number for processes with a particle system representation.
	}
	
	\section{Introduction and main result}\label{sc:intro}
	
	Multivariate Bessel processes, or radial Dunkl processes \cite{GallardoYor05,RoslerVoit98} are continuous-time Markov processes in multiple dimensions that are an interesting object of research due to their relationship with spatial symmetries, random matrices, and quantum integrable systems \cite{DunklProcesses}. They are defined as the continuous, reflection-invariant part of Dunkl processes \cite{RoslerVoit98}, which themselves are defined as an extension of $N$-dimensional Brownian motion obtained by using Dunkl operators \cite{Dunkl89} instead of partial derivatives for the semigroup generator of the process. Perhaps the most important feature of Dunkl operators is that they are differential-difference operators with difference terms that depend on the action of a finite reflection group generated by the choice of a set of vectors known as a root system. This means that for every root system there exists a corresponding multivariate Bessel process, and many well-known processes can be obtained by choosing the appropriate root system.
	
	For instance, by choosing $N=1$ and the rank-one root system, one receives the classical Bessel process of index $\nu$ described by the stochastic differential equation (SDE)
	\[\ud Z_\nu(t)=\ud B(t)+\frac{2\nu+1}{2}\frac{\ud t}{Z_\nu(t)}.\]
	Similar SDEs are obtained in higher dimensions for all other root systems, which motivates the nomenclature \emph{multivariate} Bessel processes. Note that the drift term in this SDE is singular, and that its magnitude changes depending on the value of the index $\nu$. In general, the magnitude of the singular drift is given by the multiplicity function $k$, which for particular values and root systems provides an alternative formulation of matrix eigenvalue processes. When the root system is of type $A$ or $B$ in $N$-dimensional space, the processes obtained are the Dyson model \cite{Dyson62A} and the square root of Wishart-Laguerre processes \cite{Bru91,KonigOConnell01} respectively, and if $k=1/2,$ 1, or 2 they correspond to eigenvalue processes of orthogonal, Hermitian, or symplectic random matrices respectively, with independent Brownian motions as entries up to symmetry \cite[Chapter 3]{DunklProcesses}. 
	
	An important feature of multivariate Bessel processes is that they are defined in a subset of $\mbb{R}^N$ called the Weyl chamber $W$, which is a cone with a boundary $\partial W$ given by a series of hyperplanes, and that their drift terms explode if processes hit $\partial W$. In spite of this singular drift, it is known that the SDEs of multivariate Bessel processes have unique strong solutions for any positive value of $k$ \cite{CepaLepingle, Chybiryakov, GraczykMalecki}. Moreover, it is known that when $k$ is equal to or larger than $1/2$ these processes never hit the boundary, a fact which applies to all matrix-eigenvalue process cases. However, when a multiplicity lies strictly between $0$ and $1/2$ the time it takes for the processes to hit a boundary for the first time is almost surely finite \cite{Chybiryakov, Demni09}. Regarding this last point, we showed in our previous paper \cite{HufnagelAndraus21} that in the $A$ case, namely the Dyson model, the set of collision times has a Hausdorff dimension equal to
	\[\max\Big\{0,\frac12-k\Big\}.\]
	This result came from a direct derivation where we made use of the machinery in \cite{LiuXiao98} as well as an asymptotic formula by Graczyk and Sawyer \cite[Theorem 7]{GraczykSawyer} to derive the upper bound of the Hausdorff dimension. However, there are no known generalizations of the Graczyk-Sawyer formula outside of the $A$ case, and our best attempt at extending this result to the $B$ case was the derivation of the lower bound of the Hausdorff dimension by one of the authors in \cite[Appendix A]{Hufnagel}.
	
	Fortunately, during private communications with J. Małecki, we found out that we can avoid using the Graczyk-Sawyer formula in the $A$ case by using an indirect approach: it turns out that the SDE of the squared Vandermonde determinant of the process is a time-changed version of a classical squared Bessel process, so one can obtain the same result by using the known result on the Hausdorff dimension of hitting times to the origin for the classical Bessel process. This approach is readily generalized to all reduced root systems by observing that the Vandermonde determinant in the $A$ case is the alternating polynomial of the symmetric group, the reflection group generated by the root system of type $A$, so one can immediately guess that for other root systems it is enough to study their corresponding alternating polynomial. In this paper, we consider the multivariate Bessel process of rational type $X(t)$ associated to an arbitrary reduced root system $R$ with a positive multiplicity function $k:R\to(0,\infty)$, and we investigate the Hausdorff dimension of the times when it collides with the boundary of the Weyl chamber. Our main result is the following statement.
	\begin{theorem}\label{th:main}
		The Hausdorff dimension of collision times of a multivariate Bessel process with its Weyl chamber's boundary is given by
		\[\dim\big(X^{-1}(\partial W)\big)=\max\Big\{0,\frac12-\min_{\alpha\in R}k(\alpha)\Big\}\]
		almost surely.
	\end{theorem}
	
	As mentioned in our previous paper \cite{HufnagelAndraus21}, the significance of this result is that it characterizes the collision times between particles in several one-dimensional, multiple-particle stochastic processes and gives a clear delimitation for the behavior of these processes. That is, it indicates whether they show collisions or not, and if they do one can interpret the Hausdorff dimension as a quantity that expresses how frequent these collisions are. It must be noted that our result is independent of the dimensionality of the root system considered. This dimensionality corresponds to the number of particles in the Dyson model and the Wishart-Laguerre processes, so we expect for this result to hold in the thermodynamic limit, namely, when the number of particles tends to infinity. In addition, this result is applicable to other processes which have no particle-system interpretation, such as those that correspond to the dihedral root systems in two dimensions, as well as the exceptional root systems.
	
	In order to prove our statement, we first consider the uniform multiplicity case, that is, the case where the multiplicity function is constant, and we study the SDE of the squared alternating polynomial evaluated at $X(t)$. We observe that the alternating polynomial cancels if and only if $X(t)$ hits the boundary, and we find that a random bi-Lipschitz time change on compact intervals transforms the resulting SDE into that of a classical squared Bessel process. Consequently, the Hausdorff dimension of hitting times at $\partial W$ is the same as that of a classical squared Bessel process of index $k-1/2$ hitting the origin almost surely. 
	
	For the non-uniform case, we derive coinciding upper and lower bounds for the Hausdorff dimension as follows. For the upper bound, we recall calculations made in the uniform case and replace the multiplicity function by its minimum. We find an inequality which shows that the squared alternating polynomial in the non-uniform case is bounded below by a process given by the SDE of the uniform case with the minimum multiplicity. As the latter is a time-changed classical squared Bessel process of index $\min_{\alpha\in R}k(\alpha)-1/2$, we use the monotonicity property of the Hausdorff dimension to conclude that the upper bound is given by the rhs of Theorem~\ref{th:main}. For the lower bound, we extend the idea found in the proof of \cite[Proposition 1]{Demni09}. Namely, we consider the SDE of the inner product of $X(t)$ with a simple root $\beta$ of $R$ (see Section~\ref{sc:setting}), and show that this inner product is bounded above by a classical Bessel process of index $k(\beta)-1/2$. Finally, using the countable stability and monotonicity properties of the Hausdorff dimension, we find that the Hausdorff dimension is also bounded below by the rhs of Theorem~\ref{th:main}.
	
	The paper is structured as follows. In Section~\ref{sc:setting}, we summarize the mathematical objects and notions we use in the rest of the paper. In Section~\ref{sc:uniform}, we derive our main result for the uniform multiplicity case. Note that this result applies directly to cases where the multiplicity function only takes one value, such as the $A$ case. Finally, we complete the proof of our main result in Section~\ref{sc:non-uniform} by considering the non-uniform case, which corresponds to root systems for which the multiplicity function may take more than one value, such as the $B$ case.
	
	\section{Setting, definitions, and properties}\label{sc:setting}
	\subsection{Multivariate Bessel process}	
	Before defining the main object of study in this paper, we summarize several related objects which will be referred to in later sections. For two vectors $x,y\in\R^N$ we denote the canonical inner product by $\langle x,y\rangle$, and for $x\neq0$ we describe the reflection operator $\sigma_x$ by the formula
	\[\sigma_xy:=y-2\frac{\langle x,y\rangle}{\langle x,x\rangle}x,\]
	which represents the vector obtained by reflecting the vector $y$ across the hyperplane specified by the vector $x$. Then, a \emph{root system} $R$ is a finite set of vectors, called \emph{roots}, such that if $\alpha,\beta\in R$, then $\sigma_\alpha\beta\in R$, that is, the reflection of a root by any root is also a root. Additionally, we will require the root system to be \emph{reduced}: if for $c\in\R$ both $\alpha$ and $c\alpha$ are roots, then $c=\pm1$.
	
	Some details about the structure of reduced root systems will be important in the upcoming sections. One can determine a \emph{positive subsystem} $R_+$ by choosing an arbitrary vector $u$ such that $\langle\alpha,u\rangle\neq0$ for any root $\alpha$:
	\[R_+:=\{\alpha\in R\cond{}\langle\alpha,u\rangle>0\}.\]
	For simplicity, we will keep the dependence on $R$ of the following objects implicit. 
	Now that the positive subsystem is fixed, we can define the \emph{Weyl chamber} $W$:
	\[W:=\{x\in\mathbb{R}^N\cond{}\langle\alpha,x\rangle\geq0{}\ \forall\alpha\in R_+\}.\]
	We denote the \emph{simple system} of $R_+$ by $\Delta$; this is a particular set of roots, called \emph{simple roots}, which forms a vector basis for $R_+$ in such a way that every root $\alpha\in R_+$ can be written in the form
	\begin{align}
		\label{eq:simple_roots}\alpha=\sum_{\gamma\in\Delta}c_\gamma\gamma
	\end{align}
	with all coefficients $c_\gamma\geq0$ \cite[Theorem, p. 8]{Humphreys}. It is known that if $\gamma,\xi\in\Delta$ with $\gamma\neq\xi$, then $\langle\gamma,\xi\rangle\leq0$ \cite[Corollary, p. 9]{Humphreys}. All roots not in $R_+$ are in the \emph{negative subsystem} $-R_+$ and hence are a linear combination of the roots in $\Delta$ with all coefficients non-positive. Finally, note that the reflections obtained from the roots in $R$ generate a reflection group denoted by $G$.  
	
	Associated to $R$ is a \emph{multiplicity function} $k:R\to\mathbb{C}$, which is a set of parameters assigned arbitrarily to each disjoint $G$-orbit of the roots in $R$. In this sense, we say that the multiplicity function $k(\alpha)$ is invariant under the action of $G$, that is, $k(\rho\alpha)=k(\alpha)$ for all $\alpha\in R$ and $\rho\in G$. Within our setting, we will require that the multiplicity function be real and positive, and we will refer to its values as \emph{multiplicities}. Additionally, we define the \emph{weight function}
	\begin{align}
		\label{eq:weight_function}
		w_k(x):=\prod_{\alpha\in R_+}\langle\alpha,x\rangle^{2k(\alpha)}
	\end{align}
	and the Selberg integral
	\[c_k:=\int\limits_{W}e^{-x^2/2}w_k(x)\intd x.\]
	Our study is focused on the multivariate Bessel process associated to the root system $R$, which we denote by $\{ X(t)\}_{t\geq0}$. This process is given by the semigroup generator, see \cite[p. 114, Eq. (7) and p. 199, Eq. (1)]{DunklProcesses}:
	\[\mathcal{L}_{k}:=\sum_{i=1}^N\frac{\partial^2}{\partial x_i^2}+2\sum_{\alpha\in R_+}\frac{k(\alpha)}{\langle\alpha,x\rangle}\sum_{i=1}^N\alpha_i\frac{\partial}{\partial x_i}.\]
	We denote the transition probability of $\{X(t)\}_{t\geq 0}$ ending at $y\in W$ after a time $t>0$ from the starting point $x$ by $p_k(t,y\cond{}x)$. This function satisfies the relation
	\[\frac{\partial}{\partial t}p_k(t,y\cond{}x)=\frac{1}{2}\mathcal{L}_kp_k(t,y\cond{}x).\]
	For clarity, we emphasize that $\mathcal{L}_k$ acts on the variable $x$. Indicating the $2$-norm by $\|x\|=\sqrt{\inprod{x}{x}}$, we can write $p_k(t,y\cond{}x)$ explicitly as follows, see \cite{Rosler98}, \cite[p.199, Eq. (2)]{DunklProcesses} and \cite[p. 160]{GallardoYor05}:
	\begin{align}
		\label{eq:density}
		p_k(t,y\cond{}x)= \frac{e^{-\frac{\|x\|^2+\|y\|^2}{2t}}}{c_kt^{\frac{N}{2}}}D_k\left(\frac{x}{\sqrt{t}},\frac{y}{\sqrt{t}}\right)w_k
		\left(\frac{y}{\sqrt{t}}\right).
	\end{align}
	We do not define the Dunkl-Bessel function $D_k(x,y)$ explicitly in this paper, but we point out the fact that it is a deformation of the exponential function as seen in \cite[6.6 The $\kappa$-Analogue of the Exponential, Definition 6.6.4]{DunklXu} in the sense that, when all multiplicities are zero, it satisfies the relation
	\[D_k(x,y)\big|_{k\equiv 0}=\frac{1}{|G|}\sum_{\rho\in G}\exp \langle \rho x,y\rangle.\]
	Explicit forms of this function for several root systems can be found in \cite{BakerForrester97} and in \cite[Chapter 3]{DunklProcesses}. In addition, it obeys the useful inequalities \cite[Proposition 5.1, Corollaries 5.2 and 5.3]{Rosler99}:
	\begin{align*}
		e^{-\|x\|\|y\|}\leq D_k(x,y)\leq e^{\|x\|\|y\|}.
	\end{align*}
	Given a standard Brownian motion $\{B(t)\}_{t\geq0}$ in $\R^N$, the stochastic differential equation (SDE) of $X(t)$ in vector form is
	\begin{align}
		\ud X(t)&=\ud B(t)+\sum_{\alpha\in R_+}k(\alpha)\frac{\alpha}{\langle\alpha, X(t)\rangle}\intd t.\label{eq:MultiBesselSDE}
	\end{align}
	We emphasize that this singular SDE has a unique strong solution if $k(\alpha)>0$ for every $\alpha\in R_+$, see \cite[Proposition 3.1]{Chybiryakov} and \cite[Theorem 1]{Demni09}. In the case $0<\min_{\alpha\in R_+}k(\alpha)<1/2$ this process almost surely hits the boundary $\partial W$ \cite[Proposition 1]{Demni09}. These are continuous Markov processes by \cite[Theorem 4.10]{RoslerVoit98}.
	
	\subsection{Hausdorff dimension - Definition and properties}
	The Hausdorff dimension generalizes the familiar notion of dimension. This means that well-known geometric objects like straight lines, hyperplanes and others with intuitive dimensionality keep the same dimension. The Hausdorff dimension offers a finer distinction, since it admits positive real numbers. 
	
	We denote by $B(x,R):=\{y\in \mbb{R}^N:\|y\|\leq R\}$ the closed $N$-dimensional ball centered at $x$ with radius $R$. For the monotonically increasing monomial (on the positive halfline) of power $\alpha\geq 0 $ the Hausdorff measure is specified by 
	\begin{align*}
		m_\alpha(E):=\lim\limits_{\varepsilon\to 0} \inf\Big\{\sum_{i=1}^\infty (2r_i)^\alpha: E\subset \bigcup\limits_{i=1}^\infty B(x_i,r_i); r_i<\varepsilon\Big\}.
	\end{align*}
	The radius $r_i$ may be equal to zero, hence the covering of $E$ by balls can be a finite union. The Hausdorff dimension is defined by the following lemma. 
	\begin{lemma}[{\cite[8.1 Hausdorff dimension]{Adler1981}}]
		\label{lm:Hausdorff_dimension}
		For any set $E\subset \mathbb{R}^n$ there exists a unique number $\alpha^\star$, called the Hausdorff dimension of $E$, for which 
		\begin{align*}
			\alpha <\alpha^\star \Rightarrow m_\alpha(E)=\infty, \quad \alpha >\alpha^\star \Rightarrow m_\alpha(E)=0. 
		\end{align*}
		This number is denoted by $\dim (E)$ and satisfies
		\begin{align*}
			\alpha^\star=\dim(E)=\sup\{\alpha>0:m_\alpha(E)=\infty \}=\inf \{\alpha>0:m_\alpha(E)=0\}.
		\end{align*}
	\end{lemma}
	The Hausdorff dimension fulfills the following properties.
	\begin{itemize}
		\item \textit{Countable stability \cite[2.2 Hausdorff dimension]{Falconer2013fractal}:} If $F_1,F_2,\dots $ is a (countable) sequence of sets then $$\dim \Big(\bigcup\limits_{i=1}^\infty F_i\Big)=\sup\limits_{i\in\mathbb{N}}\dim (F_i).$$
		\item\textit{Monotonicity \cite[2.2 Hausdorff dimension]{Falconer2013fractal}:} If $E\subset F$ then $\dim (E)\leq \dim (F)$.
		\item\textit{Invariance under bi-Lipschitz mapping \cite[Corollary 2.4]{Falconer2013fractal}:} If $f:E\to \R^N$ is a bi-Lipschitz mapping, that is, there exist constants $c_1\geq c_2>0$ such that for every $x,y\in E$ 
		$$c_1\|x-y\|\leq \|f(x)-f(y)\|\leq c_2\|x-y\|,$$
		then $\dim(E)=\dim(f(E))$.
	\end{itemize}
	\section{Uniform case}\label{sc:uniform}
	In this section, we analyze the multivariate Bessel process $\{X(t)\}_{t\geq 0}$ in the uniform case, that is, $k(\alpha)=k$ for all $\alpha\in R_+$ and its times hitting the boundary of the Weyl chamber, namely $X^{-1}(\partial W)$. We focus on the Hausdorff dimension of these times and we want to apply a well-known result of Liu and Xiao that determines the Hausdorff dimension of the times when a self-similar process hits the origin, see \cite[4. Hausdorff dimension of the zero set]{LiuXiao98}. In particular, the machinery in \cite{LiuXiao98} covers the case of a classical Bessel process hitting the origin. 
	
	Hence, we seek a suitable space-time transformation of the multivariate Bessel process, which turns it into a squared classical Bessel process, such that we can apply the results in \cite{LiuXiao98}. For this goal, we first define the squared alternating polynomial:
	\begin{align}
		V(x)&:=\prod_{\alpha\in R_+}\langle\alpha, x\rangle^2. 
		\label{eq:squared_alternating_polynomial}
	\end{align}
	Note that if $V(x)=0$, then there exists $\alpha\in R_+$ such that $\langle\alpha,x\rangle=0$, which implies that $x\in\partial W$ by definition; as the converse is also true, it follows that $X(t)$ hits $\partial W$ if and only if $V(X(t))=0$. Next, we calculate the derivatives of $V(x)$:
	\begin{align}
		\frac{\partial}{\partial x_i}V(x)&= 2V(x)\sum_{\alpha\in R_+}\frac{\alpha_i}{\langle\alpha, x\rangle},\notag\\
		\frac{\partial^2}{\partial x_i^2}V(x)&= 4V(x)\bigg(\sum_{\alpha\in R_+}\frac{\alpha_i}{\langle\alpha, x\rangle}\bigg)^2-2V(x)\sum_{\alpha\in R_+}\frac{\alpha_i^2}{\langle\alpha, x\rangle^2}\notag\\
		&= 4V(x)\bigg(\sum_{\alpha\in R_+}\frac{\alpha_i^2}{\langle\alpha, x\rangle^2}\notag
		+\sum_{\substack{\alpha,\beta\in R_+\\\alpha\neq\beta}}\frac{\alpha_i}{\langle\alpha, x\rangle}\cdot\frac{\beta_i}{\langle\beta, x\rangle}\bigg)-2V(x)\sum_{\alpha\in R_+}\frac{\alpha_i^2}{\langle\alpha, x\rangle^2}\notag\\
		&= 2V(x)\sum_{\alpha\in R_+}\frac{\alpha_i^2}{\langle\alpha, x\rangle^2}+4V(x)\sum_{\substack{\alpha,\beta\in R_+\\\alpha\neq\beta}}\frac{\alpha_i}{\langle\alpha, x\rangle}\cdot\frac{\beta_i}{\langle\beta, x\rangle}\notag\\
		&= 2V(x)\sum_{\alpha\in R_+}\frac{\alpha_i^2}{\langle\alpha, x\rangle^2}.\notag
	\end{align}
	By \cite[Lemma 6.4.6]{DunklXu}, the second term in the second-to-last line cancels, yielding the last line. Since $X(t)\in \partial W=\{y\in W\cond{}  \exists\alpha\in R_+:\langle \alpha, y\rangle=0\}$ holds if and only if $V(X(t))=0$, we then examine the times when the transformed process hits the origin. Below, the space transformation $V(X(t))$ will be analyzed in detail. 
	\begin{lemma}
		\label{lm:time_change}
		Assume the multiplicity function is uniform, that is, $k(\alpha)=k$ for all $\alpha\in R_+$, and define the function
		\begin{align}
			S(x)&:=V(x)\sum_{\alpha\in R_+}\frac{\|\alpha\|^2}{\langle\alpha, x\rangle^2}.\label{eq:TimeChangeDer}
		\end{align}
		Then, the process $\{Y(t)\}_{t\geq 0}:=\{V(X(t))\}_{t\geq0}$ is equivalent in distribution to a classical squared Bessel process of index $\nu=k-1/2$ with time change $s^{-1}$ given through
		\begin{align}
			s(t):=\int\limits_0^tS(X(\tau))\intd\tau.\label{eq:RandomTimeChange}
		\end{align}
	\end{lemma}
	\begin{proof}
		First, we use the It\^o formula to calculate the SDE of $Y(t)$:
		\begin{align}
			\ud Y(t)&= 2V(X(t))\sum_{\alpha\in R_+}\frac{\langle\alpha,\ud X(t)\rangle}{\langle\alpha, X(t)\rangle}+V(X(t))\sum_{\alpha\in R_+}\frac{\|\alpha\|^2}{\langle\alpha, X(t)\rangle^2}\intd t\notag\\
			&\stackrel{\eqref{eq:MultiBesselSDE}}{=} 2V(X(t))\sum_{\alpha\in R_+}\frac{1}{\langle\alpha, X(t)\rangle}\Big\langle\alpha,\ud B(t)+\sum_{\beta\in R_+}k(\beta)\frac{\beta}{\langle\beta, X(t)}\intd t\Big\rangle\notag\\
			&\quad
			+V(X(t))\sum_{\alpha\in R_+}\frac{\|\alpha\|^2}{\langle\alpha, X(t)\rangle^2}\intd t\notag\\
			&= 2V(X(t))\sum_{\alpha\in R_+}\frac{\langle\alpha,\ud B(t)\rangle}{\langle\alpha, X(t)\rangle}+2V(X(t))\sum_{\alpha,\beta\in R_+}k(\beta)\frac{\langle\alpha,\beta\rangle}{\langle\alpha, X(t)\rangle\langle\beta, X(t)\rangle}\intd t\notag\\
			&\quad+V(X(t))\sum_{\alpha\in R_+}\frac{\|\alpha\|^2}{\langle\alpha, X(t)\rangle^2}\intd t\notag\\
			&= 2V(X(t))\sum_{\alpha\in R_+}\frac{\langle\alpha,\ud B(t)\rangle}{\langle\alpha, X(t)\rangle}+2V(X(t))\sum_{\substack{\alpha,\beta\in R_+\\\alpha\neq\beta}}k(\beta)\frac{\langle\alpha,\beta\rangle}{\langle\alpha, X(t)\rangle\langle\beta, X(t)\rangle}\intd t\notag\\
			&\quad+V(X(t))\sum_{\alpha\in R_+}(2k(\alpha)+1)\frac{\|\alpha\|^2}{\langle\alpha, X(t)\rangle^2}\intd t.\notag
		\end{align}
		A slight modification of \cite[Lemma 6.4.6]{DunklXu} yields that the last term on the second-to-last line vanishes, and we obtain
		\begin{align}
			\ud Y(t)&= 2V(X(t))\sum_{\alpha\in R_+}\frac{\langle\alpha,\ud B(t)\rangle}{\langle\alpha, X(t)\rangle}+V(X(t))\sum_{\alpha\in R_+}(2k(\alpha)+1)\frac{\|\alpha\|^2}{\langle\alpha, X(t)\rangle^2}\intd t.
			\label{eq:calculation_of_V(X)}
		\end{align}
		Using the assumption $k(\alpha)=k$, the SDE becomes
		\begin{align}
			\ud Y(t)&= 2V(X(t))\sum_{\alpha\in R_+}\frac{\langle\alpha,\ud B(t)\rangle}{\langle\alpha, X(t)\rangle}+(2k+1)V(X(t))\sum_{\alpha\in R_+}\frac{\|\alpha\|^2}{\langle\alpha, X(t)\rangle^2}\intd t.\notag\\
			&= 2Y(t)\sum_{\alpha\in R_+}\frac{\langle\alpha,\ud B(t)\rangle}{\langle\alpha, X(t)\rangle}+(2k+1)Y(t)\sum_{\alpha\in R_+}\frac{\|\alpha\|^2}{\langle\alpha, X(t)\rangle^2}\intd t\notag
			\\
			&\stackrel{\eqref{eq:TimeChangeDer}}{=}2Y(t)\sum_{\alpha\in R_+}\frac{\langle\alpha,\ud B(t)\rangle}{\langle\alpha, X(t)\rangle}+(2k+1)S(X(t))\intd t.\label{eq:newSDE_Y}
		\end{align}
		Clearly, this is a semimartingale. Now, we would like to rewrite the martingale part in terms of one Brownian motion instead of a linear combination of them, and for that we define the process $\{W(t)\}_{t\geq 0}$ by $W(0)=0$ and
		\begin{align*}
			W(t):=\int_0^t\frac{\sqrt{Y(\tau)}}{\sqrt{S(X(\tau))}}\sum_{\alpha\in R_+}\frac{\langle\alpha,\ud B(\tau)\rangle}{\langle\alpha,X(\tau)\rangle}.
		\end{align*}
		This prompts us to rewrite the martingale term as follows:
		\begin{align}
			\label{eq:martingale_part_Y}
			2Y(t)\sum_{\alpha\in R_+}\frac{\langle\alpha,\ud B(t)\rangle}{\langle\alpha, X(t)\rangle}=2\sqrt{Y(t)}\sqrt{S(X(t))}\intd W(t).
		\end{align}
		To see that $\{W(t)\}_{t\geq 0}$ is a Brownian motion, we derive the quadratic variation $[W]_t$: 
		\begin{align*}
			[W]_t&=\bigg[\int_0^\cdot\frac{\sqrt{Y(\tau)}}{\sqrt{S(X(\tau))}}\sum_{\alpha\in R_+}\frac{\langle\alpha,\ud B(\tau)\rangle}{\langle\alpha,X(\tau)\rangle}\bigg]_t\\
			&=\int_0^t\frac{Y(\tau)}{S(X(\tau))}\bigg[\sum_{\alpha\in R_+}\frac{\langle\alpha,\ud B\rangle}{\langle\alpha,X\rangle}\bigg]_\tau\\
			&=\int_0^t\frac{Y(\tau)}{S(X(\tau))}\sum_{\alpha,\beta\in R_+}\frac{\langle\alpha,\beta\rangle}{\langle\alpha, X(\tau)\rangle\langle\beta, X(\tau)\rangle}\intd \tau\\
			&=\int_0^t\frac{Y(\tau)}{S(X(\tau))}\sum_{\alpha\in R_+}\frac{\|\alpha\|^2}{\langle\alpha, X(\tau)\rangle^2}\intd \tau\\
			&\stackrel{\eqref{eq:TimeChangeDer}}{=}\int_0^t\intd \tau=t. 
		\end{align*}
		In the third line we applied the well-known covariation of the Brownian motion, $[B_i,B_j]_t=\delta_{ij}t$, whereas the fourth line holds, again, by \cite[Lemma 6.4.6]{DunklXu}.
		Since $\{X(t)\}_{t\geq0}$ is a Markov process, $\{Y(t)\}_{t\geq0}=\{V(X(t))\}_{t\geq0}$ as an image under a continuous mapping is also a Markov process. It follows that $\{W(t)\}_{t\geq0}$ has independent increments and therefore it is a Brownian motion.
		Combining \eqref{eq:newSDE_Y} and \eqref{eq:martingale_part_Y}, we can write
		\begin{align}
			\ud Y(t)&=2\sqrt{Y(t)}\sqrt{S(X(t))}\intd W(t)+(2k+1)S(X(t))\intd t.
			\label{eq:new_SDE_for_V(X)}
		\end{align}
		Now, we consider the random time change $s^{-1}$, see \eqref{eq:RandomTimeChange}. In \Cref{lm:bijection} below we will show that $s$ is a bijection almost surely and hence the time change is well-defined.
		Then, we define the process $\{\widetilde{Y}(t)\}_{t\geq0}$ by
		\[\widetilde{Y}(t):=Y(s^{-1}(t)).\]
		It follows by \cite[Theorem 8.5.7]{oksendal2003} that $\{\widetilde{Y}(t)\}_{t\geq0}$ is given by the SDE
		\[\ud \widetilde{Y}(t)=2\sqrt{\widetilde{Y}(t)}\intd \widetilde W(t)+(2k+1)\intd t,\]
		where $\{\widetilde W(t)\}_{t\geq 0}$ is a Brownian motion.\footnote{Following the lines of \cite[Section 8.5 Random Time Change]{oksendal2003} we choose as volatility $\sigma(x):=2\sqrt{x}$ and drift $ b(x):=2k+1$. The corresponding functions for the time change are $ c(t):=S(X(t)), \beta(t):=\int_0^tc(t)=s(t)$ and $\alpha(t):=s^{-1}(t).$}
		This is nothing but a classical squared Bessel process of dimension $2k+1$, or index $\nu=k-1/2$, as claimed. 
	\end{proof}
	The result above implies that the Hausdorff dimension of the hitting times to the origin of a classical squared Bessel process coincides with that of the process $\{V(X(t))\}_{t\geq0}$ provided \eqref{eq:RandomTimeChange} is a bijection and a bi-Lipschitz mapping on any compact interval almost surely. We prove the first of these two statements below.
	\begin{lemma}
		The time change $s(t)=\int_0^tS(X(\tau))\intd\tau$ is a bijection on $[0,\infty)$ almost surely. \label{lm:bijection}
	\end{lemma}
	\begin{proof}
		For the proof we recall the definition of $S$: 
		\begin{align}
			\begin{split}
				\label{eq:proof_time_change_bijection}
				S(x)&=V(x)\sum_{\alpha\in R_+}\frac{\|\alpha\|^2}{\langle\alpha, x\rangle^2}, \\
				V(x)&=\prod_{\alpha\in R_+}\langle\alpha, x\rangle^2.
			\end{split}
		\end{align}
		The multivariate Bessel process is almost surely continuous \cite[Theorem 4.10]{RoslerVoit98}.  $S$ is a polynomial and hence we recognize that the time change $s$ is a continuous mapping almost surely. 
		
		In the following, we will show for every $t>0$ that $\{\tau\in[0,t]: S(X(\tau))=0\}$ forms a set of Lebesgue measure zero almost surely. We then conclude that the time change is strictly monotonously increasing almost surely, as it is given by the integral over $S(X(t))$ and $S(X(t))\geq 0$, see \eqref{eq:proof_time_change_bijection}. Combined with $s(0)=0$ and the almost sure continuity, the statement that $s$ is almost surely a bijection on $[0,\infty)$ follows. For $0<t<\infty$ we define 
		\begin{align*}
			\Theta(t):=\int\limits_0^t\mathbf{1}_{\{S(X(\tau))=0\}}\intd \tau.
		\end{align*}
		Our aim is to prove $\Theta(t)=0$ almost surely.
		Its expectation is given by
		\begin{align*}
			0\leq \mbb{E}[\Theta(t)]&=\mbb{E}\left[\int\limits_0^t\mathbf{1}_{\{S(X(\tau))=0\}}\intd \tau\right]
			=\int\limits_0^t\mbb{E}\left[\mathbf{1}_{\{S(X(\tau))=0\}}\right]\intd \tau\\
			&=\int\limits_0^t\mbb{P}\left(S(X(\tau))=0\right)\intd \tau\stackrel{\eqref{eq:proof_time_change_bijection}}{\leq}
			\int\limits_0^t\Pbb\big(V(X(\tau)) = 0\big)\intd \tau\\
			&\stackrel{\eqref{eq:proof_time_change_bijection}}{=}\int\limits_0^t\mbb{P}\left(\exists \beta\in R_+: \langle\beta,X(\tau)\rangle=0\right)\intd \tau.
		\end{align*}
		To bound the probability in the last line, we define the boundary associated to the root $\beta\in R_+$ by 
		\begin{align}
			\label{eq:boundary_specific_root}
			\partial W^\beta:=\{y\in W\cond{}  \langle \beta, y\rangle=0\}
		\end{align}
		with $$\bigcup\limits_{\beta\in R_+}\partial W^\beta=\partial W.$$
		Hence, we derive
		\begin{align*}
			0\leq \mbb{E}[\Theta(t)] &\leq \int\limits_0^t\mbb{P}\left(\exists \beta\in R_+: \langle\beta,X(\tau)\rangle=0\right)\intd \tau\\
			&=\int\limits_0^t\mathbb P \big(\exists \beta\in R_+: X(\tau)\in \partial W^\beta\big)\intd \tau\\
			&\leq \sum\limits_{\beta\in R_+}\int\limits_0^t\mathbb P \big(X(\tau )\in \partial W^\beta\big)\intd \tau.
		\end{align*}
		Next, we insert the explicit form of the density: 
		\begin{align*}
			0\leq \mbb{E}[\Theta(t)]
			&\leq \sum\limits_{\beta\in R_+}\int\limits_0^t\mathbb P \big(X(\tau)\in \partial W^\beta\big)\intd \tau\\
			&\stackrel{\eqref{eq:density}}{=}\sum\limits_{\beta\in R_+} \int\limits_0^tc_k^{-1}\int\limits_{\partial W^\beta}e^{-\frac{\|x\|^2+\|y\|^2}{2\tau^{N/2}}}D_{k}\Big(\frac{x}{\sqrt{\tau}},\frac{y}{\sqrt{\tau}}\Big)w_k\Big(\frac{y}{\sqrt{\tau}}\Big)\ud y\ud \tau\\
			&\stackrel{\eqref{eq:weight_function}}{=}\sum\limits_{\beta\in R_+}\int\limits_0^tc_k^{-1}\int\limits_{\partial W^\beta}e^{-\frac{\|x\|^2+\|y\|^2}{2\tau^{N/2}}}D_{k}\Big(\frac{x}{\sqrt{\tau}},\frac{y}{\sqrt{\tau}}\Big)\underbrace{\prod\limits_{\alpha\in R_+}\Big\langle \frac{y}{\sqrt{\tau}},\alpha\Big\rangle^{2k}}_{=0 \text{ since } \langle y, \beta\rangle=0}\ud y\intd \tau\\
			&=0.
		\end{align*}
		This indicates that the times $\tau\in[0,t]$ for which $S(X(\tau))=0$ form a set of Lebesgue measure zero almost surely. Therefore, the time change is strictly monotonously increasing on every compact interval $[0,t]$ almost surely with $0<t<\infty$ and hence has an inverse almost surely. 
\end{proof}

We are now ready to prove our result for the uniform multiplicity case; we also show that the time change $s^{-1}(t)$ is bi-Lipschitz on compact intervals as one of the crucial steps in the proof. 
\begin{theorem}\label{th:main_uniform_case}
	The Hausdorff dimension of collision times for the multivariate Bessel process with uniform multiplicities, that is $k(\alpha)=k$ for all $\alpha\in R_+$, is given by
	\begin{equation}
		\dim\big( X^{-1}(\partial W)\big)=\max\Big\{0,\frac{1}{2}-k\Big\}\notag
	\end{equation}
	almost surely.
\end{theorem}
\begin{proof}
	The case $k\geq 1/2$ is trivial since then the multivariate Bessel process never hits $\partial W$ almost surely.
	In \Cref{lm:time_change} we examined $Y(t):=V(X(t))$ by using the squared alternating polynomial \eqref{eq:squared_alternating_polynomial}. We can immediately recognize that this step is beneficial by the fact that $X(t)\in \partial W=\{y\in W\cond{}  \exists \alpha\in R_+:\langle \alpha, y\rangle=0\}$ if and only if $Y(t)=0$. In particular, we conclude that $\dim\big( X^{-1}(\partial W)\big)=\dim\big(Y^{-1}(0)\big).$ Accordingly, we focus on showing
	$$ \dim\big(Y^{-1}(0)\big)=\frac12 -k$$
	almost surely. This proof is based on the time change and the well-known Hausdorff dimension of the times when a classical Bessel process hits the origin. Initially, we argue that $s$, the inverse of our time change \eqref{eq:RandomTimeChange}, is bi-Lipschitz on any compact interval $[0,n]$ almost surely. The multivariate Bessel process $X$ is almost surely continuous and $S$ is a polynomial, thus $S(X)$ is almost surely continuous. Consequently, the integral $s$ is continuously differentiable almost surely and thereby bi-Lipschitz on any compact interval almost surely, that is, 
	\begin{align*}
		C_1|x-y|\leq |s(x)-s(y)| \leq C_2|x-y|
	\end{align*}
	for some $C_1(n),C_2(n)> 0$ and every $x,y\in[0,n]$. Applying this formula to $x:=s^{-1}(\tilde x)$ and $y:=s^{-1}(\tilde y)$, we immediately obtain the bi-Lipschitz property for $s^{-1}$ on any compact interval almost surely. 
	
	In \Cref{lm:time_change} we have shown that $Y(s^{-1}(t))$ is a classical squared Bessel process and the result 
	\begin{align}
		\label{eq:HD_for_bessel}
		\dim\Big(\big(Y(s^{-1})\big)^{-1}(0)\Big)=\frac12-k
	\end{align} 
	is known for this process almost surely. In \cite[4. Hausdorff dimension of the zero set]{LiuXiao98} or for a detailed proof \cite[4.2 Existing results on return times of a classical Bessel process]{Hufnagel} this result is shown for a classical Bessel process hitting the origin and in particular it makes no difference looking at the square. In the following, we transfer the dimension result to the set $Y^{-1}(0)$. Since $s^{-1}$ is bi-Lipschitz on any compact interval almost surely, we consider
	\begin{align}
		\begin{split}
			\label{eq:proof_HD_reformulation}
			\dim\Big(\big(Y(s^{-1})\big)^{-1}(0)\cap[0,n]\Big)&=\dim\Big(s\big(Y^{-1}(0)\big)\cap[0,n]\Big)\\
			&=\dim\bigg(s^{-1}\Big(s\left(Y^{-1}(0)\right)\cap[0,n]\Big)\bigg)\\
			&=\dim\Big(s^{-1}\left(s\big(Y^{-1}(0)\big)\right)\cap s^{-1}([0,n])\Big)\\
			&=\dim\Big(Y^{-1}(0)
			\cap s^{-1}([0,n])\Big)
		\end{split}
	\end{align}
for arbitrary $n\in \N$ almost surely.
In the second line we applied the invariance of the Hausdorff dimension under the almost sure bi-Lipschitz mapping $s^{-1}$ on $[0,n]$. By using the countable stability property twice and $$\bigcup\limits_{n\in\N}s^{-1}([0,n])=[0,\infty)$$ for the almost sure bijection $s$ on $[0,\infty)$, we conclude
\begin{align*}
	\frac12-k&\stackrel{\eqref{eq:HD_for_bessel}}{=}\dim\Big(\big(Y(s^{-1})\big)^{-1}(0)\Big)
	\\&=\dim\bigg(\bigcup\limits_{n\in\N}\big(Y(s^{-1})\big)^{-1}(0)\cap[0,n]\bigg)\\
	&=\sup\limits_{n\in \N}\dim\Big(\big(Y(s^{-1})\big)^{-1}(0)\cap[0,n]\Big)\\
	&\stackrel{\eqref{eq:proof_HD_reformulation}}{=}\sup\limits_{n\in \N}\dim\Big(Y^{-1}(0)
	\cap s^{-1}([0,n])\Big)\\
	&=\dim\bigg(\bigcup\limits_{n\in\N}Y^{-1}(0)
	\cap s^{-1}([0,n])\bigg)\\
	&=\dim\big(Y^{-1}(0)\big)
\end{align*}
almost surely.
\end{proof}

\section{Non-uniform case}\label{sc:non-uniform}
In this section, we investigate the multivariate Bessel process $\{X_t\}_{t\geq0}$ having a non-uniform multiplicity function $k:R_+\to (0,\infty)$. We consider again the Hausdorff dimension of the times hitting the boundary of the Weyl chamber $X^{-1}(\partial W)$. We divide the proof into two parts, where we trace the problem back to the known result of the classical Bessel process hitting the origin. For the upper bound of the Hausdorff dimension, we utilize the ideas of the uniform case by bounding our multiplicities by the minimum over the multiplicities and then applying the uniform case to $\min\limits_{\alpha\in R_+}k(\alpha)$. 
\begin{theorem}[Upper bound of the Hausdorff dimension]
The Hausdorff dimension of collision times for the multivariate Bessel process is bounded by
\begin{equation}
	\dim\big( X^{-1}(\partial W)\big)\leq\max\Big\{0,\frac12- \min\limits_{\alpha\in R_+}k(\alpha)\Big\}\notag
\end{equation}
almost surely.
\end{theorem}
\begin{remark}
Note that the minimum in this theorem can also be taken over $R$ because $k$ is $G$-invariant, namely, $k(\alpha)=k(\sigma_\alpha\alpha)=k(-\alpha).$ The same can be said of all other such statements below.
\end{remark}
\begin{proof}
We focus on the non-trivial case where $\min_{\alpha\in R_+}k(\alpha)<1/2$. Since we want to employ the ideas from the uniform case, we recall the process $Y(t):=V(X(t))$ from \Cref{lm:time_change} and the  calculations in the proof. In particular, instead of the uniformness of the multiplicity function we now apply $k(\beta)\geq \min\limits_{\alpha\in R_+}k(\alpha)$ in \eqref{eq:calculation_of_V(X)} resulting in a formula similar to \eqref{eq:new_SDE_for_V(X)}:
\begin{align*}
	\ud Y(t)&\geq 2\sqrt{Y(t)}\sqrt{S(X(t))}\intd W(t)+\big(2\min\limits_{\alpha\in R_+}k(\alpha)+1\big)S(X(t))\intd t
\end{align*}
with a Brownian motion $\{W(t)\}_{t\geq 0}$, as specified in the proof of \Cref{lm:time_change}. 
Considering the same random time change $s^{-1}$ and defining the process $\{\widetilde{Y}(t)\}_{t\geq0}$ by
\[\widetilde{Y}(t):=Y(s^{-1}(t)),\]
it follows by identical arguments that $\widetilde{Y}(t)$ satisfies the inequality
\begin{align}
	\label{eq:inequality_proof_nonuniform_upper_bound}
	\ud \widetilde{Y}(t)\geq2\sqrt{\widetilde{Y}(t)}\intd \widetilde W(t)+\big(2\min\limits_{\alpha\in R_+}k(\alpha)+1\big)\intd t,
\end{align}
where $\{\widetilde W(t)\}_{t\geq 0}$ is a Brownian motion. The rhs is, again, a classical squared Bessel process of dimension $2\min\limits_{\alpha\in R_+}k(\alpha)+1$, or index $\nu=\min\limits_{\alpha\in R_+}k(\alpha)-1/2$.  

With these preliminary considerations, we can prove the upper bound of the Hausdorff dimension. Let us denote the process defined through the SDE on the rhs of \eqref{eq:inequality_proof_nonuniform_upper_bound} by $\{Z(t)\}_{t\geq 0}$ with $\widetilde Y(0)=Z(0)$, namely
\begin{align*}
	\ud Z(t)=2\sqrt{Z(t)}\intd \widetilde{W}(t) +\big(2\min\limits_{\alpha\in R_+}k(\alpha)+1\big)\intd t.
\end{align*} 
The result  $$\dim\big(Z^{-1}(0)\big)=\frac12-\min\limits_{\alpha\in R_+}k(\alpha)$$ 
is well-known for a classical squared Bessel process almost surely. From \eqref{eq:inequality_proof_nonuniform_upper_bound} and \cite[Exercise 2.19, p.~294]{karatzas2014brownian}, we have $\widetilde Y(t)\geq Z(t)$ for all $t\geq 0$ almost surely, and conclude by the monotonicity of the Hausdorff dimension combined with $Z(t)\geq 0$ that 
$$\dim\big(\widetilde Y^{-1}(0)\big)\leq\dim\big(Z^{-1}(0)\big)=\frac12-\min\limits_{\alpha\in R_+}k(\alpha)$$ 
almost surely. Using the same arguments from the proof of \Cref{th:main_uniform_case}, that is, $Y(t)=0$ if and only if $X(t)\in \partial W$ and the invariance of the Hausdorff dimension under an almost-sure bi-Lipschitz mapping, we obtain 
$$\dim\big(X^{-1}(\partial W)\big)=\dim\big( Y^{-1}(0)\big)=\dim\big(\widetilde Y^{-1}(0)\big)\leq \frac12-\min\limits_{\alpha\in R_+}k(\alpha)$$
almost surely.
\end{proof}
For the lower bound, we focus on one specific root and the multivariate Bessel process hitting the boundary associated to this root. In formulas, for $\beta$ in the simple system $\Delta\subset R_+$ we can look at the process $\langle X(t), \beta \rangle\geq0$. Following the ideas of \cite{Demni09}, we prove that a normalized version of this process is a classical Bessel process with an additional, negative drift term, and hence is bounded above by a classical Bessel process.
\begin{lemma}
\label{lm:non_uniform_lower_bound}
The process $\langle X(t), \beta \rangle/\|\beta\|$ is bounded above by a classical Bessel process of index $k(\beta)-1/2$.
\end{lemma}
\begin{proof}
We define 
\begin{align}
	Y_\beta(t):=\frac{\langle X(t), \beta \rangle}{\|\beta\|}\notag
\end{align}
and compute the corresponding SDE using the It{\^o} formula:
\begin{align}
	\ud Y_\beta(t)&\stackrel{\eqref{eq:MultiBesselSDE}}{=}\frac{\langle \ud B(t), \beta \rangle}{\|\beta\|}+\sum_{\alpha\in R_+}\frac{k(\alpha)}{\|\beta\|}\frac{\langle\beta,\alpha\rangle}{\langle\alpha,X(t)\rangle}\ud t\notag\\
	&=\frac{\langle \ud B(t), \beta \rangle}{\|\beta\|}+\frac{k(\beta)}{Y_\beta(t)}\intd t+\sum_{\substack{\alpha\in R_+\backslash\{\beta\}}}\frac{k(\alpha)}{\|\beta\|}\frac{\langle\beta,\alpha\rangle}{\langle\alpha,X(t)\rangle}\ud t\notag\\
	&=\frac{\langle \ud B(t), \beta \rangle}{\|\beta\|}+\frac{k(\beta)}{Y_\beta(t)}\intd t+\sum_{\substack{\alpha\in R_+\backslash\{\beta\}\\\langle\alpha,\beta\rangle>0}}\frac{k(\alpha)}{\|\beta\|}\frac{\langle\beta,\alpha\rangle}{\langle\alpha,X(t)\rangle}\ud t+\sum_{\substack{\alpha\in R_+\backslash\{\beta\}\\\langle\alpha,\beta\rangle<0}}\frac{k(\alpha)}{\|\beta\|}\frac{\langle\beta,\alpha\rangle}{\langle\alpha,X(t)\rangle}\ud t\notag\\
	&=\frac{\langle \ud B(t), \beta \rangle}{\|\beta\|}+\frac{k(\beta)}{Y_\beta(t)}\intd t+\sum_{\substack{\alpha\in R_+\backslash\{\beta\}\\\langle\alpha,\beta\rangle>0}}\frac{k(\alpha)}{\|\beta\|}\frac{\langle\beta,\alpha\rangle}{\langle\alpha,X(t)\rangle}\ud t-\sum_{\substack{\alpha\in R_+\backslash\{\beta\}\\\langle\alpha,\beta\rangle<0}}\frac{k(\alpha)}{\|\beta\|}\frac{\langle\beta,\sigma_\beta\alpha\rangle}{\langle\alpha,X(t)\rangle}\intd t.\label{eq:simpleRootSDE}
\end{align}
The short calculation $\inprod{\beta}{\sigma_\beta\alpha}=-\inprod{\beta}{\alpha}$ justifies the last equality. The first two terms result in a classical Bessel process, so now we  simplify the two remaining sums so that we recognize that the additional drift term is negative. Once we have shown this, the statement follows immediately. 
We want to modify the second sum to achieve the condition $\inprod{\beta}{\alpha}>0$ and combine it with the first sum. In order to do this, we prove that, if $\alpha\in R_+\backslash\{\beta\}$ with $\inprod{\beta}{\alpha}<0$, then $\sigma_\beta\alpha\in R_+$ with $\inprod{\beta}{\sigma_\beta \alpha}>0$.
Since $\alpha\in R_+$, there exist constants $c_\gamma\geq0$ for every $\gamma\in\Delta$, see \eqref{eq:simple_roots}, such that
\begin{align}
	\label{eq:proof_alpha}
	\alpha=\sum_{\gamma\in\Delta}c_\gamma\gamma.
\end{align}
Then, we compute $\sigma_\beta\alpha$ as follows:
\begin{align*}
	\sigma_\beta\alpha&=\sum_{\gamma\in\Delta}c_\gamma\gamma-2\sum_{\gamma\in\Delta}\frac{\langle c_\gamma\gamma,\beta\rangle}{\langle\beta,\beta\rangle}\beta\\
	&=\sum_{\gamma\in\Delta\backslash\{\beta\}}c_\gamma\gamma-\left[c_\beta+\frac{2}{\langle\beta,\beta\rangle}\sum_{\gamma\in\Delta\backslash\{\beta\}}c_\gamma\langle\gamma,\beta\rangle\right]\beta.
\end{align*}
Because $\sigma_\beta\alpha$ is a root, it must be an element of either $R_+$ or $-R_+$, so all coefficients on the rhs are either non-negative or non-positive. Let us take a look at \eqref{eq:proof_alpha}: Since $\alpha\neq\beta$, at least one of the coefficients $c_\gamma$ must be positive for $\gamma\neq \beta$. Therefore, $\sigma_\beta\alpha\in R_+$.  With this in mind, we can rewrite the last term in \eqref{eq:simpleRootSDE} by replacing the dummy variable $\alpha$ by $\sigma_\beta\alpha$ and applying $\sigma_\beta\big(\sigma_\beta \alpha\big)=\alpha$:
\begin{align*}
	\sum_{\substack{\alpha\in R_+\backslash\{\beta\}\\\langle\alpha,\beta\rangle<0}}\frac{k(\alpha)}{\|\beta\|}\frac{\langle\beta,\sigma_\beta\alpha\rangle}{\langle\alpha,X(t)\rangle}\intd t&=\sum_{\substack{\sigma_\beta\alpha\in R_+\\\langle\sigma_\beta\alpha,\beta\rangle<0}}\frac{k(\sigma_\beta\alpha)}{\|\beta\|}\frac{\langle\beta,\alpha\rangle}{\langle\sigma_\beta\alpha,X(t)\rangle}\intd t\\
	&=\sum_{\substack{\alpha\in R_+\backslash\{\beta\}\\\langle\alpha,\beta\rangle>0}}\frac{k(\alpha)}{\|\beta\|}\frac{\langle\beta,\alpha\rangle}{\langle\sigma_\beta\alpha,X(t)\rangle}\intd t.
\end{align*}
Since   $\alpha=\beta$ implies $\sigma_\beta\alpha=-\beta\notin R_+$, the rhs on the first line has a sum with the condition $\sigma_\beta\alpha\in R_+$ instead of $\sigma_\beta\alpha\in R_+\backslash\{\beta\}$.
The second line follows from the $G$-invariance of the multiplicity function, that is, $k(\sigma_\beta\alpha)=k(\alpha)$ for every $\alpha,\beta\in R_+$, and the equality $\langle\beta,\sigma_\beta\alpha\rangle=\langle\sigma_\beta\beta,\alpha\rangle=-\langle\beta,\alpha\rangle$. Inserting this expression into \eqref{eq:simpleRootSDE}, we obtain
\begin{align}
	\ud Y_\beta(t)&=\frac{\langle \ud B(t), \beta \rangle}{\|\beta\|}+\frac{k(\beta)}{Y_\beta(t)}\intd t+\sum_{\substack{\alpha\in R_+\backslash\{\beta\}\\\langle\alpha,\beta\rangle>0}}\frac{k(\alpha)}{\|\beta\|}\bigg[\frac{\langle\beta,\alpha\rangle}{\langle\alpha,X(t)\rangle}-\frac{\langle\beta,\alpha\rangle}{\langle\sigma_\beta\alpha,X(t)\rangle}\bigg]\intd t\notag\\
	&=\frac{\langle \ud B(t), \beta \rangle}{\|\beta\|}+\frac{k(\beta)}{Y_\beta(t)}\intd t-\sum_{\substack{\alpha\in R_+\backslash\{\beta\}\\\langle\alpha,\beta\rangle>0}}\frac{k(\alpha)}{\|\beta\|}\langle\beta,\alpha\rangle\frac{\langle\alpha-\sigma_\beta\alpha,X(t)\rangle}{\langle\alpha,X(t)\rangle\langle\sigma_\beta\alpha,X(t)\rangle}\intd t\notag\\
	&=\frac{\langle \ud B(t), \beta \rangle}{\|\beta\|}+\frac{k(\beta)}{Y_\beta(t)}\intd t-\sum_{\substack{\alpha\in R_+\backslash\{\beta\}\\\langle\alpha,\beta\rangle>0}}2\frac{k(\alpha)}{\|\beta\|}\frac{\langle\beta,\alpha\rangle^2}{\langle\beta,\beta\rangle}\frac{\langle\beta,X(t)\rangle}{\langle\alpha,X(t)\rangle\langle\sigma_\beta\alpha,X(t)\rangle}\intd t\notag\\
	&=\frac{\langle \ud B(t), \beta \rangle}{\|\beta\|}+\frac{k(\beta)}{Y_\beta(t)}\intd t-Y_\beta(t)\sum_{\substack{\alpha\in R_+\backslash\{\beta\}\\\langle\alpha,\beta\rangle>0}}\frac{2k(\alpha)\langle\beta,\alpha\rangle^2}{\|\beta\|^2\langle\alpha,X(t)\rangle\langle\sigma_\beta\alpha,X(t)\rangle}\intd t\notag.
\end{align}
All terms in the sum are clearly positive, meaning that the last term in this expression is negative. Rewriting slightly, we obtain the SDE
\begin{align}
	\ud Y_\beta(t)&=\ud B_\beta(t)+\frac{k(\beta)}{Y_\beta(t)}\intd t-Y_\beta(t)\sum_{\substack{\alpha\in R_+\backslash\{\beta\}\\\langle\alpha,\beta\rangle>0}}\frac{2k(\alpha)\langle\beta,\alpha\rangle^2}{\|\beta\|^2\langle\alpha,X(t)\rangle\langle\sigma_\beta\alpha,X(t)\rangle}\intd t\notag,
\end{align}
with $B_\beta(t):=\langle B(t),\beta\rangle/\|\beta\|$ denoting a Brownian motion. This reveals that $Y_\beta(t)$ can be understood as a classical Bessel process of index $k(\beta)-1/2$ with an additional  drift term that pushes it towards the origin. With this in mind, we define the Bessel process $Z_\beta(t)$ by the SDE
\begin{align}
	\ud Z_\beta(t)=\ud B_\beta(t)+\frac{k(\beta)}{Z_\beta(t)}\intd t,\quad Z_\beta(0)=Y_\beta(0).\notag
\end{align}
With these considerations we immediately conclude $Y_\beta(t)\leq Z_\beta(t)$ almost surely \cite[Exercise 2.19, p.~294]{karatzas2014brownian}. 

\end{proof}
\begin{theorem}[Lower bound of the Hausdorff dimension]
The Hausdorff dimension of collision times for the multivariate Bessel process is bounded by
\begin{equation}
	\dim\big( X^{-1}(\partial W)\big)\geq\max\Big\{0,\frac12- \min\limits_{\alpha\in R_+}k(\alpha)\Big\}\notag
\end{equation}
almost surely.
\end{theorem}
\begin{proof}
First of all, we focus on one specific root and examine the boundary associated to this root. When applying the previous lemma, we must argue that we can cover the entire boundary with hyperplanes defined by simple roots. Hence, we establish the following:
\begin{align}
	\label{eq:lowerbound_simpleroots}
	\partial W
	&=\bigcup_{\alpha\in R_+} \partial W^\alpha\
	=\bigcup_{\beta\in \Delta}\partial W^\beta.
\end{align}
Recall that $\partial W^\beta$ is defined in \eqref{eq:boundary_specific_root}. The first equation is trivial. The relation $\supseteq$ in the second equation is true since $R_+
\supseteq \Delta$. The relation $\subseteq$ is valid because the simple system $\Delta$ forms a vector basis for $R_+$ with positive coefficients, see \eqref{eq:simple_roots}. In particular, if $\alpha\in R_+$ then there exist coefficients $c_\beta\geq 0$ for  every $\beta\in\Delta$ such that 
\begin{align}
	\label{eq:alpha_representation}\alpha=\sum\limits_{\beta\in\Delta}c_\beta \beta.
\end{align}
If now $x\in \partial W^\alpha$, that is, $\langle \alpha,x\rangle=0$, then 
\begin{align*}
	0=\langle \alpha, x \rangle =\sum\limits_{\beta\in\Delta}\underbrace{c_\beta}_{\geq0} \underbrace{\langle\beta,x\rangle}_{\geq 0}.
\end{align*}
Hence, $c_\beta\langle\beta,x\rangle=0$ for every $\beta\in\Delta$. If $ x$ is not the null vector (which is the trivial case), there exists at least one $\beta\in\Delta$ such that $c_\beta\neq0$ in \eqref{eq:alpha_representation}. Thus, $x\in \partial W^\beta=\{y\in W\cond{}  \langle \beta, y\rangle=0\}$ and the equality is established. Using the countable stability property of Hausdorff dimensions on the finite set $\Delta$, we conclude 
\begin{align*}
	\dim\big( X^{-1}(\partial W)\big)\stackrel{\eqref{eq:lowerbound_simpleroots}}{=}\dim\bigg(\bigcup_{\beta\in\Delta} X^{-1}(\partial W^\beta)\bigg)=\max\limits_{\beta\in\Delta}\Big\{\dim\big( X^{-1}(\partial W^\beta)\big)\Big\}.
\end{align*}
In \Cref{lm:non_uniform_lower_bound} we proved that $$\frac{\langle X(t),\beta\rangle}{\|\beta\|}\geq 0$$ is bounded above by a Bessel process $Z_\beta$ of index $k(\beta)-1/2$. Therefore, if $Z_\beta(t)=0$, then $\langle X(t),\beta\rangle=0$, i.e. $X^{-1}(\partial W^\beta)\supseteq Z_\beta^{-1}(0)$. Using the monotonicity of the Hausdorff dimension we receive 
\begin{align*}
	\dim\big( X^{-1}(\partial W)\big)=\max\limits_{\beta\in\Delta}\Big\{\dim\big( X^{-1}(\partial W^\beta)\big)\Big\}&\geq \max\limits_{\beta\in\Delta}\Big\{\dim\big( Z_\beta^{-1}(0)\big)\Big\}\\
	&=\max\limits_{\beta\in\Delta}\Big\{\max\Big\{0,\frac12-k(\beta)\Big\}\Big\}\\
	&=\max\Big\{0,\frac12-\min\limits_{\beta\in\Delta} k(\beta)\Big\}\\
	&=\max\Big\{0,\frac12-\min\limits_{\alpha\in R_+} k(\alpha)\Big\}
\end{align*}
almost surely. We again used the well-known Hausdorff dimension of the times when a classical Bessel process hits the origin. 
The last line follows from \cite[Corollary, p. 11]{Humphreys}, which implies that for every root $\alpha\in R_+$ there exist $\rho\in G$ and $\beta\in\Delta$ such that $\rho\alpha=\beta$; then, due to the $G$-invariance of $k$, we see that $k(\alpha)=k(\rho\alpha)=k(\beta)$, and therefore the minimum of $k$ taken over $\Delta$ is equivalent to that taken over $R_+$.
\end{proof}

\section*{Acknowledgments and Funding} 
The authors wish to thank J. Małecki for fruitful and enlightening discussions, which were the foundation for this paper. NH would like to thank M. Jakubzik for helpful comments on the proofs in the uniform case, and SA would like to thank N. Hatano for discussions regarding the alternating polynomial. SA is supported by the JSPS Kakenhi Grant number JP19K14617.

\bibliography{bibtex}
\end{document}